\documentclass[12pt]{amsart}
\usepackage[left=1.5in,right=1.5in]{geometry}
\usepackage{amsmath}
\usepackage{amsthm}
\usepackage{amssymb}
\usepackage{todonotes}
\usepackage{tikz}

\newcommand{\RR}{\mathbb{R}}
\newcommand{\PP}{\mathbb{P}}
\newcommand{\mc}{\mathcal}
\DeclareMathOperator{\spn}{span}
\DeclareMathOperator{\vol}{vol}
\usepackage[utf8]{inputenc}

\DeclareMathOperator{\adj}{adj}

\newtheorem{theorem}{Theorem}[section]
\newtheorem{proposition}[theorem]{Proposition}
\newtheorem{corollary}[theorem]{Corollary}

\theoremstyle{definition}
\newtheorem{definition}[theorem]{Definition}
\newtheorem{example}[theorem]{Example}

\theoremstyle{remark}
\newtheorem*{remark}{Remark}

\title{Canonical forms of polytopes from adjoints}
\author{Christian Gaetz}
\address{Department of Mathematics, University of California, Berkeley, CA, USA}
\email{gaetz@berkeley.edu}
\date{February 2020}

\begin{document}
\maketitle

\begin{abstract}
Projectivizations of pointed polyhedral cones $C$ are positive geometries in the sense of Arkani-Hamed, Bai, and Lam \cite{AHBL}.  Their canonical forms look like
\[
\Omega_C(x)=\frac{A(x)}{B(x)} dx,
\]
with $A,B$ polynomials.  The denominator $B(x)$ is just the product of the linear equations defining the facets of $C$. We will see that the numerator $A(x)$ is given by the \emph{adjoint} polynomial of the dual cone $C^{\vee}$.  The adjoint was originally defined by Warren \cite{Warren} who used it to construct barycentric coordinates in general polytopes.

Confirming our intuition that the job of the numerator is to cancel unwanted poles outside the polytope, we will see that the adjoint is the unique polynomial of minimal degree whose hypersurface contains the \emph{residual arrangement} of non-face intersections of supporting hyperplanes of $C$.
\end{abstract}

\

These notes were prepared for a lecture given at MIT in February 2020. No originality is claimed for the results, but the details of some proofs may not have appeared elsewhere before.

\section{Setup}

Throughout, $C$ will denote a convex polyhedral cone in $\mathbb{R}^{m+1}$.  We will further assume that $C$ is \emph{pointed}, meaning that it does not contain any line.  Projective polytopes in $\mathbb{P}^m$ can be thought of as the images of pointed convex cones (with the origin removed) in $\mathbb{R}^{m+1}$ under the standard map $\mathbb{R}^{m+1} \to \mathbb{P}^m$; indeed, this is how Arkani-Hamed, Bai, and Lam \cite{AHBL} define projective polytopes.  Because of this correspondence we move freely between claims about cones $C$ and their associated projective polytopes $P$; in particular, we may write $\Omega_C$ for the canonical form $\Omega_P$ of the positive geometry $P$.  When it is convenient to work with the (affine) cross-sectional polytope of $C$, we always work with the slice $\{ (x_0,...,x_m) \in \RR^{m+1} \: | \: x_0=1\} \cap C$, and we always assume that this cross-sectional slice of $C$ contains the ``origin" $(1,0,\ldots,0)$ (if not, we may change chart on $\PP^m$). 

We write $V(C)$ for the set of unit vectors generating the vertex rays of $C$, the rays which are not in the convex hull of any other pair of rays from the cone.  These rays are the rays through the vertices of the cross-sectional polytope of $C$.

The \emph{dual cone} $C^{\vee}$ of $C$ is
\[
C^{\vee}=\{x \in \RR^{m+1} \: | \: x \cdot y \geq 0, \forall y \in C\}.
\]
It is not hard to see that $F \mapsto F^{\vee}$ is an inclusion reversing bijection from the faces of $C$ to the faces of $C^{\vee}$.  The \emph{dual polytope} of $P$ in $\RR^m = \{x_0=1\}$ is 
\[
P^{\vee}=\{(x_1,...,x_m) \: | \: x \cdot y \leq 1, \forall y \in P\}.
\]
Note that the $x_0=1$ cross section of $C^{\vee}$ is $-P^{\vee}$; this sign will reappear later in the slightly different conventions for adjoints of cones and of polytopes.

By a \emph{triangulation} of $C$, we mean a collection $T$ of simplicial cones of dimension $\dim(C)$ such that 
\begin{itemize}
    \item $\bigcup_{S \in T} S = C$,
    \item each intersection $S \cap S'$ of cones in $T$ is a face of both $S$ and $S'$, and 
    \item $V(S) \subseteq V(C)$ for all $S \in T$.
\end{itemize}
This last condition is not always required of triangulations, but will be important for thinking about adjoints.

\section{Adjoints of cones and polytopes}

For a simplicial cone $S$, we let $a_S$ denote the volume of the parallelepiped determined by the (length one) vertex rays $V(S)$ of $S$.  The following definition is due to Warren \cite{Warren}, who introduced it in order to describe barycentric coordinates on polytopes.

\begin{definition} \label{def:adjoint}
Let $C$ be a pointed convex cone in $\RR^{m+1}$ with triangulation $T$, the \emph{adjoint} of $C$ is the polynomial in $x_1,...,x_{m+1}$ defined by
\[
\adj_C(x)=\sum_{S \in T} a_S \prod_{v \in V(C)\setminus V(S)} (v \cdot x).
\]
A priori this polynomial depends on the triangulation $T$, but, as we will see shortly, it is in fact independent of $T$.
\end{definition}

\begin{remark}
The reader is warned that \cite{Kohn-Ranestad} and \cite{Warren} use different conventions for how $\adj_C$ is normalized, so the polynomials differ by a constant factor in the two papers.  The convention used in \cite{Kohn-Ranestad} is that 
\begin{equation} \label{eq:adjoint-of-polytope}
\adj_P=\sum_{\sigma \in \tau(P)} \vol(\sigma) \prod_{v \not \in \sigma} (1-v_1x_1 - \cdots -v_mx_m),
\end{equation}
where $\tau(P)$ is the triangulation of $P$ obtained from a triangulation of $C$ by intersecting with $\{x_0=1\}$, and where $(v_1,...,v_m)=v$ are coordinates in $\{x_0=1\} \cong \RR^m$.

When only interested in determining $\adj_C$ up to a constant, one may take \emph{any} (not necessarily normalized) set $V$ of vertex rays and use Definition \ref{def:adjoint}.  Since each vertex ray appears once in each summand (either in the term $a_S$ or in the product), the resulting polynomial is only changed by a scalar multiple.
\end{remark}

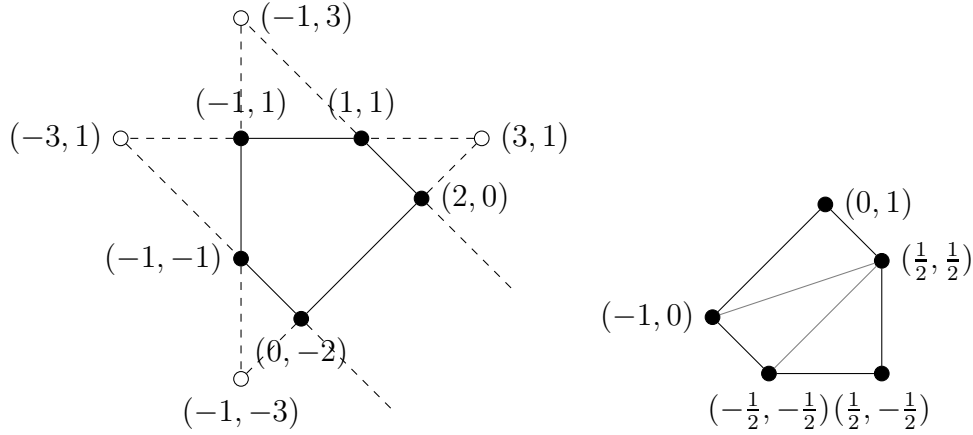
\begin{figure} \label{fig:pentagon}
\begin{tikzpicture}[scale=0.8]
\node[draw,shape=circle,fill=black,scale=0.5](a)[label=above: {$(1,1)$}] at (1,1) {};
\node[draw,shape=circle,fill=black,scale=0.5](b)[label=above: {$(-1,1)$}] at (-1,1) {};
\node[draw,shape=circle,fill=black,scale=0.5](c)[label=right: {$(2,0)$}] at (2,0) {};
\node[draw,shape=circle,fill=black,scale=0.5](d)[label=left: {$(-1,-1)$}] at (-1,-1) {};
\node[draw,shape=circle,fill=black,scale=0.5](e)[label=below: {$(0,-2)$}] at (0,-2) {};

\node[draw,shape=circle,fill=white,scale=0.5](f)[label=right: {$(3,1)$}] at (3,1) {};
\node[draw,shape=circle,fill=white,scale=0.5](g)[label=right: {$(-1,3)$}] at (-1,3) {};
\node[draw,shape=circle,fill=white,scale=0.5](h)[label=left: {$(-3,1)$}] at (-3,1) {};
\node[draw,shape=circle,fill=white,scale=0.5](i)[label=below: {$(-1,-3)$}] at (-1,-3) {};

\draw (a)--(b)--(d)--(e)--(c)--(a);
\draw[dashed] (c)--(f)--(a)--(g)--(b)--(h)--(d)--(i)--(e);
\draw[dashed] (e)--(1.5,-3.5);
\draw[dashed] (c)--(3.5,-1.5);
\end{tikzpicture}
\begin{tikzpicture}[scale=1.5]
\node[draw,shape=circle,fill=black,scale=0.5](a)[label=right: {$(0,1)$}] at (0,1) {};
\node[draw,shape=circle,fill=black,scale=0.5](b)[label=left: {$(-1,0)$}] at (-1,0) {};
\node[draw,shape=circle,fill=black,scale=0.5](c)[label=below: {$(-\frac{1}{2},-\frac{1}{2})$}] at (-0.5,-0.5) {};
\node[draw,shape=circle,fill=black,scale=0.5](d)[label=below: {$(\frac{1}{2},-\frac{1}{2})$}] at (0.5,-0.5) {};
\node[draw,shape=circle,fill=black,scale=0.5](e)[label=right: {$(\frac{1}{2},\frac{1}{2})$}] at (0.5,0.5) {};

\draw (a)--(b)--(c)--(d)--(e)--(a);
\draw[gray] (b)--(e)--(c);
\end{tikzpicture}
\caption{A pentagon $Q$ in the plane together with its supporting arrangement (left) and the dual polytope $Q^{\vee}$ with a triangulation (right).}
\end{figure}

\begin{example}
Let $Q$ be the pentagon shown in Figure \ref{fig:pentagon}, and $C$ the cone over it.  Then using (\ref{eq:adjoint-of-polytope}), up to scalars, we have 
\begin{align*}
\adj_{C^{\vee}}(x)&=\frac{1}{2}(x_0+\frac{1}{2}x_1+\frac{1}{2}x_2)(x_0-\frac{1}{2}x_1+\frac{1}{2}x_2) \\
&+ \frac{1}{2}(x_0-x_2)(x_0-\frac{1}{2}x_1+\frac{1}{2}x_2) \\
&+ \frac{1}{2}(x_0+x_1)(x_0-x_2) \\
&=\frac{3}{2}x_0^2+\frac{1}{4}x_0x_1-\frac{1}{8}x_1^2-\frac{1}{4}x_0x_2-\frac{1}{4}x_1x_2-\frac{1}{8}x_2^2.
\end{align*}
Notice that this polynomial vanishes on the intersection points of the supporting hyperplanes of $Q$ (remembering that $Q$ lives in the space $\{x_0=1\}$) including the intersection point $(0,1,-1)$ ``at infinity" of the parallel hyperplanes (see Figure \ref{fig:pentagon}).
\end{example}

Let $f(C)$ denote the set of facets of $C$.  Assuming $C$ is full-dimensional in $\RR^{m+1}$, each facet $F \in f(C)$ has a unique unit-length inward-pointing normal vector for which we write $n_F$.

\begin{theorem}[Warren \cite{Warren}] \label{thm:adjoint-in-terms-of-facets}
Let $L: \RR^{m+1} \to \RR$ be any linear function, and suppose that $C$ is full-dimensional.  Then
\begin{equation} \label{eq:adjoint-from-facets}
L(x)\adj_C(x)=\sum_{F \in f(C)} L(n_F) \adj_F(x) \prod_{v \in V(C) \setminus V(F)} (v \cdot x). 
\end{equation}
\end{theorem}
\begin{proof}
We first prove the theorem in the case $C=S$ is simplicial.  It suffices to prove the claim for $L(x)=(w \cdot x)$ for $w \in V(S)$, since these functions span $(\RR^{m+1})^*$.  Thus we need to show that 
\[
(w \cdot x)a_S = \sum_{F \in f(S)} (w \cdot n_F)a_F(v_F \cdot x),
\]
where $v_F$ is the unique element of $V(S) \setminus V(F)$.  If $w \neq v_F$, then $w \in V(F)$ and so $(w \cdot n_F)=0$.  Thus the above equation reduces to
\[
(w \cdot x)a_S = (v_F \cdot n_F)a_F(w \cdot x),
\]
and we have $a_S=(v_F \cdot n_F)a_F$ since the volume of a parallelepiped is the volume of its base times its height above that base.  Thus the theorem holds for simplicial cones. 

Now suppose that $C$ is not necessarily simplicial and let $T(C)$ be a triangulation of $C$.  Multiplying both sides by $L(x)$ in Definition \ref{def:adjoint} gives
\[
L(x) \adj_C(x) = \sum_{S \in T} L(x) a_S \prod_{v \in V(C) \setminus V(S)} (v \cdot x).
\]
Applying the simplicial case to expand $L(x) a_S=L(x) \adj_S(x)$ this becomes:
\begin{equation} \label{eq:after-sub}
L(x)\adj_C(x) = \sum_{S \in T(C)} \left( \sum_{F \in f(S)} L(n_F) a_F(x) \prod_{v \in V(C) \setminus V(F)} (v \cdot x) \right).
\end{equation}

If $F$ is an interior facet of $T(C)$, then $F$ is a facet of two cones from $T(C)$ on opposite sides of $F$; since $n_F$ is inward pointing in each of these, the corresponding terms cancel in the sum.  On the other hand, the exterior facets of cones in $T$ give triangulations $T(F)$ of each facet of $C$, so we may rewrite (\ref{eq:after-sub}) as:
\begin{align*}
L(x) \adj_C(x) & =\sum_{F \in f(C)} \left( \sum_{S \in T(F)} L(n_S) \adj_S(x) \prod_{v \in V(C) \setminus V(S)} (v \cdot x) \right) \\
&= \sum_{F \in f(C)} L(n_F) \left( \sum_{S \in T(F)}\adj_S(x) \prod_{v \in V(F) \setminus V(S)} (v \cdot x) \right) \prod_{v \in V(C) \setminus V(F)} (v \cdot x) \\
&= \sum_{F \in f(C)} L(n_F) \adj_F(x) \prod_{v \in V(C) \setminus V(F)} (v \cdot x).
\end{align*}
\end{proof}

\begin{corollary}
The polynomial $\adj_C$ does not depend on the triangulation $T$ of $C$ appearing in Definition \ref{def:adjoint}.
\end{corollary}
\begin{proof}
Taking any nonzero linear function $L$ in Theorem \ref{thm:adjoint-in-terms-of-facets} we obtain a formula for $\adj_C$ in terms of the adjoints of the facets; this may be applied recursively to compute $\adj_C$ without choosing any triangulations (since a two-dimensional cone has a unique triangulation). 
\end{proof}

\section{Adjoints and canonical forms}

We have the following expression for $\Omega_P(x)$.

\begin{theorem}[See Eq. 7.173 of \cite{AHBL}]
\label{thm:adjoint-volume-of-shifted-dual}
Let $P$ be a full-dimensional polytope in $\{x_0=1\} \subset \RR^{m+1}$, then for any $x$ in the interior of $P$ we have 
\[
\Omega_P(x)=\vol((P-x)^{\vee}) dx,
\]
where $P-x$ denotes Minkowski difference.
\end{theorem}

We will use Theorem \ref{thm:adjoint-volume-of-shifted-dual} to see that $\adj_{P^{\vee}}(x)$ is the numerator of $\Omega_P$.

\begin{theorem} \label{thm:adjoint-is-numerator}
Let $P$ be a full-dimensional polytope in $\RR^m \cong \{x_0=1\} \subset \RR^{m+1}$, then we have
\[
\Omega_P(x)=\frac{\adj_{P^{\vee}}(x)}{\prod_{F \in f(P)} (1-v_F \cdot x)} dx,
\]
where $v_F$ is the vector in $\RR^m$ such that $v_F \cdot y=1$ for all $y \in F$.
\end{theorem}
\begin{proof}
We identify $\{x_0=1\}$ with $\RR^{m}$ and use the corresponding inner product, so the inner product of two points in this plane does not reflect the fact that they both have $x_0$ coordinate equal to one.  By Theorem \ref{thm:adjoint-volume-of-shifted-dual} it suffices to show, for $x$ in the interior of $P$, that:
\begin{align}
\vol((P-x)^{\vee})&=\frac{\adj_{P^{\vee}}(x)}{\prod_{F \in f(P)} (1-v_F \cdot x)}  \\
&=\left(\sum_{\sigma \in \tau(P^{\vee})} \vol(\sigma) \prod_{v \in V(P^{\vee}) \setminus V(\sigma)} (1-v \cdot x)\right) / \left(\prod_{F \in f(P)} (1-v_F \cdot x)\right) \\
&=\sum_{\sigma \in \tau(P^{\vee})} \frac{\vol(\sigma)}{\prod_{v \in V(\sigma)} (1-v \cdot x)},
\label{eq:sum-of-simplices}
\end{align}
where $ \tau(P^{\vee})$ is some triangulation (not introducing any new vertices) of $P^{\vee}$ and where in the last step we have used the fact that the normal vectors $v_F$ for $F \in f(P)$ are exactly the vertices of $P^{\vee}$.

Now, notice that the vertices $u_F(x)$ of $(P-x)^{\vee}$ are just multiples of the vertices $v_F$ of $P^{\vee}$, since $P-x$ is just a translate of $P$.  More precisely, we have
\begin{equation} \label{eq:dilation-of-vertices}
u_F(x)=\frac{1}{1-v_F \cdot x} v_F.
\end{equation}
Since we have this natural correspondence between the vertices of $(P-x)^{\vee}$ and those of $P^{\vee}$, the triangulation $\tau(P^{\vee})$ gives a triangulation $\tilde{\tau}$ of $(P-x)^{\vee}$, and, of course, we have
\[
\vol((P-x)^{\vee})=\sum_{\tilde{\sigma} \in \tilde{\tau}} \vol(\sigma).
\]
Comparing this to (\ref{eq:sum-of-simplices}), it suffices to show that the volumes of the dilated simplices $\tilde{\sigma} \in \tilde{\tau}$ are just the volumes of the original simplex $\sigma$ times the product of the dilating factors of the vertices.  It is not true that volumes of simplices behave this way under arbitrary dilations of the vertices, but it is true for dilations of this special form, as can be seen from an elementary exercise in linear algebra (after expressing the volumes in terms of determinants).

\end{proof}

\section{Residual arrangements and adjoints}
\label{sec:residuals}

In light of Theorem \ref{thm:adjoint-is-numerator}, the material in this section is meant to confirm the intuitive idea that the numerator of $\Omega_P$ necessarily cancels unwanted poles outside of $P$ in a ``minimal" way.

\begin{proposition}
\label{prop:degree-of-adjoint}
The polynomial $\adj_C$ is homogeneous of degree $|V(C)|-\dim(C)$.
\end{proposition}
\begin{proof}
This is apparent from Definition \ref{def:adjoint} since the products
\[
\prod_{v \in V(C) \setminus V(S)} (v \cdot x)
\]
each have $|V(C)|-|V(S)|=|V(C)|-\dim(C)$ linear terms.
\end{proof}

\begin{definition}
The \emph{supporting arrangement} $\mc{H}_C$ of $C$ is the arrangement of supporting hyperplanes for the facets of $C$.  The \emph{residual arrangement} $\mc{R}_C$ is the arrangement of linear subspaces of $\RR^{m+1}$ which are intersections of hyperplanes from $\mc{H}_C$ and which do not contain any face of $C$.  In general, $\mc{R}_C$ is an arrangement of subspaces of varying dimensions.

We write $\mc{H}_P$ and $\mc{R}_P$ for the analogous arrangements associated to a projective polytope $P \subset \PP^m$; these are the images of $\mc{H}_C$ and $\mc{R}_C$ under the map $\RR^{m+1} \setminus 0 \to \PP^m$ sending $C \setminus 0$ to $P$.
\end{definition}

\begin{proposition}[Kohn and Ranestad \cite{Kohn-Ranestad}; special case due to Warren \cite{Warren}]
\label{prop:adjoint-vanishes}
The adjoint $\adj_{C}$ vanishes on the residual arrangement $\mc{R}_{C^{\vee}}$ of the dual cone.
\end{proposition}
\begin{proof}
We proceed by induction on the dimension $m+1$ of the cone $C$, the base case of dimension one being trivial, since the residual arrangement in this case is empty.

If $m+1>1$, let $R$ be an irreducible component of $\mc{R}_{C^{\vee}}$ of codimension $c$, so $R=H_1 \cap \cdots \cap H_c$ for some hyperplanes $H_i \in \mc{H}_{C^{\vee}}$.  Let $v_1,\ldots ,v_c$ be the corresponding vertices of $C$, so that $H_i$ is the hyperplane orthogonal to $v_i$.  Since $R$ is assumed not to contain a face of $C^{\vee}$, we know that $\{ v_1,\ldots,v_c \}$ is not the set of vertices of any face of $C$.

We will show that $\adj_C$ vanishes on $R$ by showing that each summand on the right hand side of Theorem \ref{thm:adjoint-in-terms-of-facets} vanishes.  Let $F \in f(C)$ be any facet.  If some $v_i \not \in V(F)$, then the term $(v_i \cdot x)$ appears in the product, and so the product vanishes on $R \subset H_i$, thus we may assume that $v_1,...,v_c \in F$ and therefore that $n_F \in R$.

Now, viewing $F$ itself as a full-dimensional cone in the $m$-dimensional space $\spn(F)=\RR^{m+1}/\RR n_F$, note that $H_i/\RR n_F$ is the supporting hyperplane of a facet of $F^{\vee}$ (since $v_i$ is a vertex of $F$).  Since $\{v_1,\ldots,v_c\}$ is not a face of $C$, it is not a face of $F$ either, and so 
\[
R'=H_1/\RR n_F \cap \cdots \cap H_c/\RR n_F=R/\RR n_F
\]
is contained in the residual arrangement $\mc{R}_{F^{\vee}}$ of $F^{\vee}$.  By induction, $\adj_F$ vanishes on $R'$, and therefore it vanishes on any element of $R' \oplus \RR n_F = R$ ($F$ is not simplicial, otherwise $\{v_1,...,v_c\} \subset F$ would be a face).  Thus $\adj_C$ vanishes on $R$ as desired.
\end{proof}

\subsection{Uniqueness of adjoints}

We say that $\mc{H}_P$ is \emph{simple} if at most $m$ hyperplanes pass through any point in $\PP^m$ (equivalently, at most $m$ hyperplanes of $\mc{H}_C$ pass through any nonzero point in $\RR^{m+1}$). 

\begin{theorem}[Kohn and Ranestad \cite{Kohn-Ranestad}] 
\label{thm:minimality-theorem}
Let $P$ be a full-dimensional polytope in $\PP^m$ with $d$ facets.  If the hyperplane arrangement $\mc{H}_P$ is simple, there is a unique hypersurface $A_P$ in $\PP^m$ of degree $d-m-1$ which contains the residual arrangement $\mc{R}_P$.  
\end{theorem}

By Proposition \ref{prop:adjoint-vanishes} and Proposition \ref{prop:degree-of-adjoint} we know that this unique hypersurface is the zero locus of $\adj_{C^{\vee}}$.

\begin{example}
Let $P \subset \PP^3$ have the combinatorial type of a $3$-cube, so $P$ has $d=6$ facets.  Consider several cases:
\begin{itemize}
    \item If $P$ is generic, so that none of the pairs of opposite facets are parallel, then $\mc{H}_P$ is simple.  In this case $\mc{R}_P$ consists of three skew lines, the intersections of the supporting hyperplanes for the pairs of opposite faces.  $A_P$ is the unique quadric (degree $6-3-1=2$) passing through these three lines, and is defined by the polynomial $\adj_{C^{\vee}}$.
    \item If $P$ is a regular cube, so that all three pairs of opposite facets are parallel, then each of these pairs of hyperplanes intersect in a line contained in the plane at infinity (that is, $x_0=0$); so $\mc{R}_P$ consists of these three lines in a plane.  Each pair of these lines must intersect, so $\mc{H}_P$ is not simple, because four hyperplane pass through such an intersection point.  In this case $\mc{R}_P$ is contained in a degree one hypersurface $\{x_0=0\}$; it is still defined as a set by $\adj_{C^{\vee}}=x_0^2$, although this is no longer a reduced scheme.
\end{itemize}
\end{example}

What can be said when $\mc{H}_P$ is not simple?

\begin{proposition}[Kohn and Ranestad \cite{Kohn-Ranestad}]
\label{prop:limits-of-simples}
Let $P$ be a full-dimensional polytope in $\PP^m$.  If $P'_t$ and $P''_t$ are continuous families of polytopes with simple hyperplane arrangements such that 
\[
\lim_{t \to \infty} P'_t = \lim_{t \to \infty} P''_t =P,
\]
then the limits of their hypersurfaces coincide:
\[
\lim_{t \to \infty} A_{P'_t} = \lim_{t \to \infty} A_{P''_t} := A_P.
\]
And $A_P$ is the zero locus of $\adj_{C^{\vee}}$, but may not be reduced.
\end{proposition}


\bibliographystyle{plain}
\bibliography{main}

\end{document}